\documentclass{amsart}
\usepackage{graphicx,amsfonts,amssymb,amsmath,amsthm}
\usepackage[pdftex]{hyperref}  
\hypersetup{pdfstartview={XYZ 120 670 1}, pdfpagemode=none}
\usepackage{cite}  
\usepackage{pdfsync}

\theoremstyle{plain} 
\newtheorem{theorem}    {Theorem}[section] 
\newtheorem*{thm}    {Theorem}
\newtheorem{lemma}      [theorem]{Lemma}
\newtheorem{corollary}  [theorem]{Corollary}

\theoremstyle{definition}

\newtheorem*{question}   {Question}
\newtheorem*{cond}    {Condition}

\theoremstyle{remark}
\newtheorem{remark}              {Remark}

\numberwithin{equation}{section}

\def\A{\mathbb A}
\def\C{\mathbb C}

\def\Q{\mathbb Q}

\begin{document}

\title[On the size of Satake parameters for GL(4)]{On the size of Satake parameters for unitary cuspidal automorphic representations for GL(4)}

\author{Nahid Walji}

\maketitle
\begin{abstract}
Let $\Pi$ be a cuspidal automorphic representation for GL(4) over a number field $F$. We obtain unconditional lower bounds on the number of places at which the Satake parameters are not ``too large''. In the case of self-dual $\Pi$ with non-trivial central character, our results imply that the set of places at which $\Pi$ is tempered has an explicit positive lower Dirichlet density. Our methods extend those of Ramakrishnan by careful analysis of the hypothetical possibilities for the structure of the Langlands conjugacy classes, as well as their behaviour under functorial lifts. We then discuss the analogous problem in GL(3).
\end{abstract}

\section{Introduction}

Let $n$ be a positive integer and $\pi$ a unitary cuspidal automorphic representation for GL(n)/$F$, where $F$ is a number field. Let $v$ denote a prime at which $\pi$ is unramified. We express the associated conjugacy class in $GL_n(\C)$ of Satake parameters as $A_v(\pi) = \{\alpha _{1,v}, \dots, \alpha _{n,v}\}$.

A fundamental problem in the theory of automorphic representations is to bound the size of these parameters. The \textit{Generalised Ramanujan Conjecture} (GRC) implies that $|\alpha_{i,v}| = 1$ for $i = 1, \dots , n$ and all $v$.
One approach towards this has been to obtain bounds that hold at every $v$, which currently stand at $|\alpha_{i,v}| \leq Nv^{\theta (n)}$ for $\theta (n) = {1}/{2} - {1}/{(n ^2 + 1)}$ for $n \geq 5$, with stronger bounds of 7/64, 5/14, and 9/22 holding for $n = 2,3$ and $4$, respectively~\cite{LRS99, Ki03, BB11}.

Here, we take a different approach, which is to seek stronger, \textit{uniform} bounds (i.e., not dependent on the norm of the prime) for a subset of primes $v$. Fix a real number $r \geq 1$ and define the set
\begin{align*}
S(r) = S(\pi,r) := \{v \mid 1/r \leq |\alpha_{i,v}| \leq r \text{ for all }i \}.
\end{align*}
Note that GRC would imply that $S (1)$ is the set of all $v$.
We now ask:
\begin{question}
What are the strongest lower bounds on the size of $S(r)$ that can be obtained? In particular, what lower bound can be established on its lower Dirichlet density?
\end{question}

For $r = 1$, there are known bounds for $n = 2$ and $n = 3$. When $n = 2$, Kim--Shahidi~\cite{KS02}, following D.~Ramakrishnan~\cite{Ra97}, showed that there exists a set of primes $v$ at which GRC holds (i.e., for such $v$ we have $|\alpha_{i,v}|= 1$ for all $i$) that has a lower Dirichlet density of 34/35. The method in~\cite{Ra97} also provides results, for $n > 2$, for the set of places at which $|\alpha_{1,v} + \dots + \alpha_{n,v}|\leq n$ has a lower Dirichlet density of at least $1- 1/n ^2$. However, this does not a priori provide any information on the sizes of the Satake parameters themselves. 

When $n = 3$ and $F = \Q$, Ramakrishnan~\cite{Ra04b} showed that there exist an infinite number of Ramanujan primes (i.e., the primes at which the Ramanujan conjecture holds) for any unitary cuspidal automorphic representation.
For $n = 4$ and any given number field $F$, there are no such results known for general $\pi$. If we consider a representation $\pi$ for GL(n)/$F$ ($n \geq 4$) that is cuspidal, algebraic, regular, self-dual, with a discrete series component at a finite place, and fix $F$ to be a totally real field or a CM field, then by Clozel~\cite{Cl91} one knows that all the primes at which $\pi$ is unramified are Ramanujan primes.

We will consider the case of $n = 4$. First, we establish the following condition to help us distinguish between various types of unitary cuspidal automorphic representations for GL(4).
\begin{cond}
We will say that a unitary cuspidal automorphic representation for GL(4) is of \textbf{type (T)} if it is either an Asai lift of a dihedral representation, an automorphic induction of a cuspidal automorphic representation for GL(2), or the Langlands tensor product of two such representations.
\end{cond}

Our first two theorems concern unitary cuspidal automorphic representations which are not of type (T). This is because those which are of type (T) are easier to handle and it is convenient for us to do so separately (see Theorem~\ref{t3}).

Note that an automorphic representation $\Pi$ for GL(n)/$F$ is said to be \textit{essentially self-dual} if there exists a Hecke character $\chi$ such that $\Pi \simeq \widetilde{\Pi} \otimes \chi$.

We shall prove: 
\begin{theorem}\label{t1}
 Let $\Pi$ be a unitary cuspidal automorphic representation for GL(4)/$F$ that is not of type (T). As above, for real $r \geq 1$, let 
 \begin{align*}
S = S(\pi,r) = \{v \mid 1/r \leq |\alpha_{i,v}| \leq r \text{ for all }i \},
 \end{align*}
and set $m = r + 1/r$. The variable $c$ takes the value 1 when $\Pi$ is essentially self-dual, and 0 otherwise.

Then 
  \begin{align*}
    \underline{\delta}(S) \geq 1- \left( \frac{1}{(m-2)^2} + \frac{1}{(2m-2-c)^2} \right),
  \end{align*}
where $\underline{\delta}(S)$ denotes the lower Dirichlet density of the set $S$.
\end{theorem}

\begin{remark}
The right-hand side of this inequality is positive when $r$ is greater than approximately 2.66.
For example, if $r = 4$ and $\Pi$ is not essentially self-dual, then there exists a set $S$ of primes of lower Dirichlet density greater than 0.77 at which none of the Satake parameters have absolute value greater than 4 or less than 1/4.
\end{remark}

In the self-dual case, we obtain stronger bounds for small values of $r$:
\begin{theorem}\label{t2}
Fix a unitary self-dual cuspidal automorphic representation $\Pi$ for GL(4)/$F$ that is not of type (T). Let $r, m$, and $S$ be defined as above. Then 
\begin{align*}
  \underline{\delta}(S) \geq   1 - \frac{2}{(m-1)^2} 
\end{align*}
when $\Pi$ has trivial central character, and
\begin{align*}
\underline{\delta}(S) \geq  \left(  \frac{1}{2} - {\rm min}\left\{ \frac{1}{2}, \frac{2}{(m-1)^2} \right\} \right) + \left(\frac{1}{2} - \frac{1}{m ^2} \right) 
\end{align*}
when $\Pi$ has quadratic central character.
\end{theorem}

\begin{remark}
For the first inequality, this provides a positive lower bound for $\underline{\delta}(S)$ when $r$ is greater than 1.88. For example, when $r = 3$ the set $S$ has a lower Dirichlet density that is greater than 0.82.
\end{remark}
The second inequality has positive right-hand side for all possible $r$ (i.e., $r \geq 1$). Furthermore:
\begin{corollary}
The set of Ramanujan primes for a unitary self-dual cuspidal automorphic representation with non-trivial central character has a lower Dirichlet density of at least 1/4. 
\end{corollary}

The cases excluded from the theorems above, namely those of type (T), can be easily handled using results from GL(2).
\begin{theorem}\label{t3}
  Let $S = S (1)$ be defined as above. 
If $\Pi$ is a unitary cuspidal automorphic representation $\Pi$ for GL(4)/$F$ that is of type (T), then 
 \begin{align*}
\underline{\delta}(S) \geq \frac{33}{70}.
 \end{align*}
In particular:
\begin{itemize}
  \item If $\Pi$ is an Asai transfer of a dihedral representation for GL(2), 
then we have $\underline{\delta}(S) = 1$.
  \item If $\Pi$ is automorphically induced from GL(2), then $\underline{\delta}(S) \geq 33/70$.
  \item If $\Pi$ is the automorphic tensor product of two cuspidal automorphic representations for GL(2), then $\underline{\delta}(S) \geq 33/35$.
\end{itemize}
\end{theorem}

Our proofs make use of a number of instances of functoriality, namely automorphic induction, Asai lifting, the automorphic tensor product, and, most importantly, the automorphy of the exterior square for GL(4) from~\cite{Ki03} (along with the cuspidality criterion of Asgari--Raghuram~\cite{AR11}). For Theorems~\ref{t1} and~\ref{t2} in particular, we make use of~\cite{Ra04b} to determine the initial approach, functoriality of the exterior square lift to find the likelihood of various types of parametrizations of Langlands conjugacy classes, and~\cite{KS02,Ra97} to then establish density bounds on various sets. \\

We then turn to the case of unitary cuspidal automorphic representations for GL(3). We will show:
\begin{theorem} \label{t4}
Let $\pi$ be a cuspidal automorphic representation for GL(3)/$F$, where $F$ is a number field.
Fix $r \geq 1$, and denote by $S_r$ the set of primes $p$ where 
\begin{align*}
  \left(\frac{r + 1}{2}\right) - \sqrt{\left(\frac{r + 1}{2}\right)^2 -1}  \leq |\alpha_{i,p}| \leq   \left(\frac{r + 1}{2}\right) + \sqrt{\left(\frac{r + 1}{2}\right)^2 -1}
\end{align*}
for all $i$.
Then 
  \begin{align*}
    \underline{\delta}(S_r) \geq 1 - \frac{1}{r^2}.
  \end{align*}
\end{theorem}
\begin{remark}
Note that as $r \rightarrow 1$, the upper and lower bounds on $\alpha_{i,p}$ both tend to 1, and the lower bound on the density tends to 0.
\end{remark}

\begin{corollary}
For $S_3$, which is the set of primes $v$ where $2 - \sqrt{3} \leq |\alpha_{i,v}|\leq 2 + \sqrt{3}$ for $i = 1,2,3$, we have
 \begin{align*}
  \underline{\delta}(S_3) \geq 8/9.\\
 \end{align*}
\end{corollary}

One can ask what happens if we try to apply our techniques to automorphic forms associated to general linear groups of higher rank. There are two issues that arise. 
The first is technical: As the rank $n$ increases the number of Satake parameters increases. This leads to a wider variety of possible structures for the Langlands conjugacy class (both for the automorphic representation itself as well as for functorial lifts of the automorphic representation). 
It is by no means guaranteed that enough of these conjugacy class structures have traces with unusual sizes or values, which is what we would have wanted to exploit.
The second is a deeper problem: Functoriality of any (non-trivial) symmetric power, exterior power, or adjoint lift is not known for GL(n) for any $n \geq 5$. This means that we cannot make use of the results of~\cite{Ra97} to address the hypothetical possibilities for the structure of the Langlands conjugacy class.\\

Our paper is organised as follows: In Section~\ref{prelim}, we briefly cover some background on lifts and transfers of automorphic representations as well as some theorems that will be needed later. In Section~\ref{GL4}, we address the case when the exterior square transfer of the representation is cuspidal, which we shall refer to as the `main case' for GL(4). This will prove Theorem~\ref{t1} in the case $c = 0$. In Section~\ref{esd}, we address the essentially self-dual case, proving Theorem~\ref{t1} for $c = 1$ as well as Theorem~\ref{t2}. In Section~\ref{T} we address the automorphic representations of type (T) and prove Theorem~\ref{t3}. Lastly, in Section~\ref{GL3}, we consider the analogous problem in GL(3), proving Theorem~\ref{t4}.

\subsection*{Acknowledgements.\\}
The author would like to thank D.~Ramakrishnan for directing my attention towards such problems and for having me as a visitor at the California Institute of Technology from January to February 2012. Thanks is due to the mathematics department there for providing me with a productive work environment. I would also like to thank A.~Raghuram for a useful discussion on the isobaric decomposition of the exterior square transfer, and F.~Brumley and P.~Nelson for their helpful comments on an earlier draft of this paper.

\section{Preliminaries} \label{prelim}

We will recall a categorisation of cuspidal automorphic representations for $GL_4(\A_F)$, where $F$ is a number field, arising from the work of M.~Asgari and A.~Raghuram.
Before doing so, let us briefly review the relevant lifts of cuspidal automorphic representations.

\subsection*{Exterior square.\\}
Given a cuspidal automorphic representation $\Pi$ for $GL_4(\A_F)$, there is a transfer to an automorphic representation $\eta$ for $GL_6(\A_F)$~\cite{Ki03} which corresponds almost everywhere to the exterior square map 
\begin{align*}
\Lambda^2: GL_4(\C) \rightarrow GL_6(\C) .
\end{align*}
In particular, at all unramified places $v$ we have 
\begin{align*}
\Lambda^2 (\Pi_v) \simeq \eta_v .
\end{align*}
In this paper we will abuse notation slightly and write $\eta$ as $\Lambda^2 (\Pi)$, referring to it as the exterior square transfer. The methods used here are not affected by the fact that we do not know if the exterior square correspondence holds locally for all the ramified places.

\subsection*{Asai lift and Langlands tensor product.\\}
Let $K$ be a quadratic algebra over $F$. Fix $G$ to be the $F$-group of GL(2)/$K$ and let $\pi$ be a cuspidal automorphic representation of $G(\A_F)$. Then there exists a transfer which corresponds to the $L$-group morphism $^LG \ \rightarrow \ ^LGL(4)$. When $K$ is $F \times F$, this is the Langlands tensor product, whose automorphy is established in~\cite{Ra00}. When $K$ is a quadratic field extension of $F$, this is known as the Asai transfer, whose automorphy is proved in~\cite{Ra02} and whose cuspidality is addressed in~\cite{Ra04} and in the appendix of~\cite{Kr12}.\\

\subsection*{Automorphic Induction.\\}
Given a cuspidal automorphic representation $\pi$ for $GL_n(\A_E)$ and a cyclic extension $E / F$ of prime degree $k$, then there exists an automorphic representation ${\rm AI}(\pi)$ for $GL_{nk}(\A_F)$ such that $L(s, {\rm AI}(\pi)) = L(s, \pi)$ and, in particular, 
\begin{align*}
  L(s, {\rm AI}(\pi)_v) = \prod_{w \mid v} L(s, \pi_w)
\end{align*}
for places $v$ of $F$ and $w$ of $E$.

We say that ${\rm AI}(\pi)$ is the automorphic induction of $\pi$ with respect to the extension $E / F$; one knows of its automorphy and of a cuspidality criterion due to~\cite{AC89}.\\

We will make repeated use of the following two theorems, the first providing a cuspidality criterion for the exterior square transfer.

\begin{theorem}[Asgari--Raghuram~\cite{AR11}] \label{ar}
  The following are equivalent:\\
(i) $\Lambda ^2 \Pi$ is not cuspidal\\
(ii) $\Pi$ is one of the following:\\
(a) $\pi_1 \boxtimes \pi_2$, where $\pi_1, \pi_2$ are cuspidal automorphic representations for $GL_2(\A_F)$.\\ 
(b) ${\rm As}(\pi)$, the Asai transfer of a dihedral cuspidal automorphic representation $\pi$ for $GL_2(\A_E)$, where $E$ is a quadratic extension of $F$.\\
(c) a functorial transfer of a cuspidal automorphic representation $\pi$ for $GSp_4(\A_F)$ (that is, it is of symplectic type).\\
(d) the automorphic induction of a cuspidal automorphic representation $\pi$ for $GL_2(\A_E)$, where $E$ is again a quadratic extension of $F$. We denote this as $\Pi = I^F_E (\pi)$.\\
(iii) $\Pi$ is either \\
(a) essentially self-dual and not the Asai transfer of a non-dihedral cuspidal automorphic representation\\
or\\
(b) admits a (non-trivial) self-twist.\\\\
Note that (ii) (a) - (c) is equivalent to (iii) (a), and (ii) (d) is equivalent to (iii) (b). 
\end{theorem}

\begin{remark}
If $\Pi$ is of type (c), then the exterior square is not cuspidal, and its isobaric decomposition is of the form 
\begin{align*}
\Lambda ^2 \Pi \simeq \widetilde{r_5} \boxplus \omega,
\end{align*}
where $\widetilde{r_5}$ is an automorphic representation of $GL_5(\A_F)$ and $\omega$ is a Hecke character.
\end{remark}

The second theorem provides upper bounds on the occurrence of large Hecke eigenvalues.
\begin{theorem}[Ramakrishnan~\cite{Ra97}]\label{dr}
Let $\pi$ be a cuspidal automorphic representation for GL(n). Then 
\begin{align*}
  \underline{\delta}(\{v \mid |a_v(\pi)|\leq k\}) \geq 1-\frac{1}{k^2}
\end{align*}
where $\underline{\delta}$ is the lower Dirichlet density.
\end{theorem}

\section{Main case for GL(4)}\label{GL4}

Here we prove Theorem~\ref{t1} for the case when $c = 0$, which we recall here:
\begin{thm}
  Let $\Pi$ be a unitary cuspidal automorphic representation for GL(4)/$F$ that is not of type (T) and that is not essentially self-dual. For real $r \geq 1$, we denote by $S = S(r)$ the set of primes $v$ at which $\Pi$ is unramified and $1/r \leq |\alpha_{i,v}|\leq r$ for all $i$, and let $m = r + 1/r$.

 Then 
  \begin{align*}
    \underline{\delta}(S) \geq 1- \left( \frac{1}{(m-2)^2} + \frac{1}{(2m-2)^2} \right),
  \end{align*}
where $\underline{\delta}(S)$ denotes the lower Dirichlet density of the set $S$.
\end{thm}

We will abuse notation slightly by letting $q$ denote both a prime of $F$ at which $\Pi$ is unramified, as well as the norm of this prime. The distinction should be clear from context.

Let $q$ be a prime at which $\Pi$ is unramified. Associated to $q$ is a conjugacy class, which we will represent by the set $A_q (\Pi) = \{\alpha, \beta, \gamma, \delta\}$ of Satake parameters. We denote the trace of this conjugacy class as $a_q(\pi)$. If $q$ is not a Ramanujan prime, then since $\Pi$ is unitary, we know that one of the parameters, $\alpha$ (say), has absolute value greater than one. We write $\alpha = u q ^t$, where $u$ is unitary and $t$ is positive.

Now $\overline{\Pi} \simeq \widetilde{\Pi}$ implies 
\begin{align*}
  A_q(\overline{\Pi}) = \{\overline{\alpha}, \overline{\beta}, \overline{\gamma}, \overline{\delta}\} = \{\alpha ^{-1}, \beta ^{-1}, \gamma ^{-1}, \delta ^{-1}\} = A_q(\widetilde{\Pi})
\end{align*}

Without loss of generality, let us say that $\beta = u q ^{-t}$.\\
There are two cases:
\begin{itemize}
  \item  \textbf{Case I.}
If $\gamma$ has absolute value 1, then so does $\delta$. Then
\begin{align*}
  A_q(\Pi) = \{u q^t , u q^{-t}, \gamma, \delta\}. 
\end{align*}
\item  \textbf{Case II.}
If $\gamma$ does not have absolute value 1, then nor does $\delta$, and we will write $\gamma = v q^s $ and $\delta = v q^{-s}$, where $v$ is unitary and $s > 0$. This gives
\begin{align*}
  A_q(\Pi) = \{u q^t, u q^{-t}, v q^s, v q^{-s}\}.
\end{align*}
\end{itemize}

Fix $r > 1$ and let $T=T(r)$ be the set of primes $q$ at which $\Pi$ is unramified and where $|\alpha_{q,i}| > r$ for some $i$. We wish to impose an upper bound on the size of $T$. Let us say that $q \in T$. Then, in Case I: 
\begin{align*}
  |a_q(\Pi)| = |uq^t + uq^{-t} + \gamma + \delta| > m-2,
\end{align*}
where $m:= r + \frac{1}{r}$. Since $\Pi$ is cuspidal, from~\cite{Ra97} we see that Case I can only occur for a set of density at most
\begin{align*} 
\frac{1}{(m-2)^2}.
\end{align*}

For Case II, we turn to the exterior square lift of $\Pi$:
\begin{align*}
 A_q(\Lambda ^2, \Pi) =& \{u ^2, uv q ^{t + s}, uv q ^{t-s}, uv q ^{-t+s}, uv q ^{-t-s}, v ^2\}
\end{align*}
so 
\begin{align*}
   |a_q(\Lambda ^2, \Pi)| =& |u ^2 + v ^2 + uv Q _t Q_s| > 2 (m-1)
\end{align*}
where $Q_a := q^a + q^{-a}$, for any real $a$.

Since $\Pi$ is not essentially self-dual and does not admit a self-twist, by Theorem~\ref{ar} we know that $\Lambda^2 \Pi$ is cuspidal. Thus Case II can only occur for a set of density at most
\begin{align*}
\frac{1}{4 (m-1)^2}.
\end{align*}

Combining the density estimates from both cases, we obtain 
\begin{align*}
  \overline{\delta}(T) \leq \frac{1}{(m-2)^2} + \frac{1}{4 (m-1)^2}.
\end{align*}

\section{Essentially self-dual GL(4)}\label{esd}

In this section we prove Theorem~\ref{t1} in the case when $c = 1$:
\begin{thm}
  Let $\Pi$ be an essentially self-dual unitary cuspidal automorphic representation for GL(4)/$F$ that is not of type (T). For real $r \geq 1$, we denote by $S = S(r)$ the set of primes $v$ at which $\Pi$ is unramified and $1/r \leq |\alpha_{i,v}|\leq r$ for all $i$, and let $m = r + 1/r$. 

 Then 
  \begin{align*}
    \underline{\delta}(S) \geq 1- \left( \frac{1}{(m-2)^2} + \frac{1}{(2m-3)^2} \right),
  \end{align*}
where $\underline{\delta}(S)$ denotes the lower Dirichlet density of the set $S$.
\end{thm}

and we also prove Theorem~\ref{t2}:
\begin{thm}
Fix a self-dual cuspidal automorphic representation $\Pi$ for GL(4)/$F$ that is not of type (T). Let $r, m$, and $S$ be as above. Then 
\begin{align*}
  \underline{\delta}(S) \geq   1 - \frac{2}{(m-1)^2} .
\end{align*}
when $\Pi$ has trivial central character, and
\begin{align*}
\underline{\delta}(S) \geq  \left(  \frac{1}{2} - {\rm min}\left\{ \frac{1}{2}, \frac{2}{(m-1)^2} \right\} \right) + \left(\frac{1}{2} - \frac{1}{m ^2} \right) .
\end{align*}
when $\Pi$ has quadratic central character.
\end{thm}

When $\Pi$ is essentially self-dual, there exists a Hecke character, that we will denote by $\eta$, such that $\widetilde{\Pi} \simeq \Pi \otimes \eta$. Let $q$ be a prime at which $\Pi$ is unramified. We will again have $q$ denote both a prime of $F$ and its norm, where context will distinguish between the two.

\subsection{Structure} We would like to determine the possible structures for $A_q(\Pi)$ when $q$ is not a Ramanujan prime for $\Pi$. Let us write $A_q(\Pi) = \{\alpha, \beta,\gamma,\delta\}$; if $q$ is not Ramanujan, then without loss of generality we can write $\alpha = u q ^t$, where $u$ is unitary and $t$ is positive.

Now 
\begin{align*}
  A_q(\Pi \otimes \eta) = A_q(\widetilde{\Pi}) = A_q(\overline{\Pi})
\end{align*}
implies 
\begin{align*}
 \{u q ^t \eta, \beta \eta, \gamma \eta, \delta \eta \} = \{\overline{u}q ^{-t}, \beta ^{-1}, \gamma ^{-1}, \delta ^{-1}\} = \{\overline{u}q ^t, \overline{\beta}, \overline{\gamma}, \overline{\delta}\}
\end{align*}
where, in the last line, and from here on, $\eta = \eta_q$.

Without loss of generality we set $u q ^t \eta = \beta ^{-1}$. We now have two main options:\\
1. If $u q ^t \eta \neq \overline{u} q ^t$, then (without loss of generality) $u q ^t \eta = \overline{\gamma}$, so we now have 
\begin{align*}
\{u q ^t \eta, \overline{u}q ^{-t}, \overline{u}q ^t, \delta \eta\} = \{\overline{u}q ^{-t}, u q ^t \eta, u q ^{-t}\eta, \delta ^{-1}\} = \{\overline{u}q ^t, u q ^{-t}\eta, u q ^t \eta, \overline{\delta}\}.
\end{align*}
Therefore, $\{\overline{u}q ^t, \delta \eta\}= \{u q ^{-t} \eta, \delta ^{-1}\}$ and $\{\overline{u}q ^{-t}, \delta ^{-1}\} = \{\overline{u} q ^t , \overline{\delta}\}$. These imply $\delta = u q ^{-t}$, so 
\begin{align} \label{case1}
  A_q(\Pi) &=  \{ u q ^t, \overline{u \eta} q ^{-t}, \overline{u \eta} q ^t, u q ^{-t}\}\\
&=  \{\alpha, \alpha ^{-1} \overline{\eta}, \overline{\alpha \eta}, \overline{\alpha ^{-1}}\}. \notag 
\end{align}
2. If $u q ^t \eta = \overline{u} q ^t$, then $\eta = \overline{u}^2$ and
\begin{align*}
\{\overline{u} q ^t, \overline{u}q ^{-t}, \gamma \eta, \delta \eta\} = \{\overline{u}q ^{-t}, \overline{u} q ^t , \gamma ^{-1}, \delta ^{-1}\} = \{\overline{u}q ^t, u q ^{-t}\eta, \overline{\gamma}, \overline{\delta}\}.
\end{align*}

Either $\gamma ^{-1} = \overline{\gamma}$ or $\gamma ^{-1} = \overline{\delta}$:

(a) If $\gamma ^{-1} = \overline{\gamma}$, then $\delta ^{-1} = \overline{\delta}$, in which case both parameters are unitary.
We are left with $\{\gamma \eta, \delta \eta \} = \{\gamma ^{-1}, \delta ^{-1}\}$.

(i) If $\gamma \eta = \gamma ^{-1}$, then $\delta \eta = \delta ^{-1}$, so $\eta = \gamma ^{-2}$ and $\gamma = \pm \delta$. Thus overall we get 
\begin{align} \label{case2}
  A_q(\Pi) = \{u q ^t, \overline{u \eta}q ^{-t}, \gamma, \pm \gamma\}.
\end{align}

(ii) If $\gamma \eta = \delta ^{-1}$, then we get 
\begin{align}\label{case3}
  A_q(\Pi) = \{u q ^t, u q ^{-t}, v, \overline{v \eta}\}
\end{align}
where $v$ is unitary.\\

(b) If $\gamma ^{-1} = \overline{\delta}$ and so $\delta ^{-1} = \overline{\gamma}$, we are left with 
\begin{align*}
  \{\gamma \eta, \gamma ^{-1} \gamma\}= \{\gamma ^{-1} , \overline{\gamma}\}.
\end{align*}

(i) If $\gamma \eta = \gamma ^{-1}$ and $\overline{\gamma ^{-1}}\eta = \overline{\gamma}$, then $\gamma ^2 \eta = 1$ and so $|\gamma|= 1$. So 
\begin{align}\label{case4}
  A_q(\Pi) = \{\alpha, \alpha ^{-1} \overline{\eta}, \gamma, \gamma\},
\end{align}
where $\gamma$ is unitary.

(ii) If $\gamma \eta = \overline{\gamma}$ and $\overline{\gamma ^{-1}}\eta = \overline{\gamma ^{-1}}$, then $\eta = 1$ and $\gamma, u$ are real.

So we get 
\begin{align}\label{case5}
  A_q(\Pi) = \{\alpha, \alpha ^{-1}, \gamma, \gamma ^{-1}\},
\end{align}
where $\alpha$ and $\gamma$ are real.

\subsection{Estimates} We now obtain upper bounds for the number of times that various cases can occur. Recall that under our assumption, $\alpha$ is not unitary; we now want to specify a bound on the size of this parameter. Fix $r > 1$ and let $T = T (r)$ and $m$ be defined as in the previous section.

\begin{itemize}
\item In Case 1, we have 
\begin{align*}
  A_q(\Lambda ^2 \Pi) = \{\overline{\eta}, \overline{\eta}q ^{2t}, u ^2, \overline{u ^2 \eta ^2}, \overline{\eta}q ^{-2t}, \overline{\eta}\},
\end{align*}
and the trace is 
\begin{align*}
  a_q(\Lambda ^2 \Pi) = \overline{\eta}(2 + Q _{2t}) + u ^2 + \overline{u ^2 \eta ^2},
\end{align*}
so $|a_q(\Lambda ^2 \Pi)| > m ^2 -2$.

\item In Case 2(a)(i),
  \begin{align*}
    |a_q(\Pi)| = |u Q _t + v \pm v| > m - 2.
  \end{align*}

\item For Case 2(a)(ii),
  \begin{align*}
    |a_q(\Pi)| = |u Q _t + (v + \overline{v \eta})| > m - 2.
  \end{align*}

\item Case 2(b)(i) is similar to Case 2(a)(i), and the same estimate holds.

\item For Case 2(b)(ii), we will write $A_q(\Pi)$ as $\{u q ^t, u q ^{-t}, v q ^s, v q ^{-s}\}$, where $u,v \in \{\pm 1\}$, $t$ is positive, and $s$ is non-negative:
\begin{align*}
  A_q(\Lambda ^2 \Pi) = \{1, uv q ^{t + s}, uv q ^{t-s}, uv q ^{-t + s}, uv q ^{-t-s}, 1\}
\end{align*}
so 
\begin{align*}
  |a_q(\Lambda ^2 \Pi)| = |2 + uv (Q_t Q_s)| > 2 (m-2). 
\end{align*}
\end{itemize}

In Cases 2(a)(i),(ii) and 2(b)(i), we have the same bound on $|a_q(\Pi)|$, and since $\Pi$ is cuspidal we can apply Theorem~\ref{dr} and conclude that the set $T$ of places at which these cases occur has the bound 
\begin{align}\label{est1}
  \overline{\delta} (T)  \leq \frac{1}{(m-2)^2}.
\end{align}

For Cases 1 and 2(b)(ii), we need to first establish the isobaric decomposition of the exterior square lift.
\begin{lemma}
Let $\Pi$ be an essentially self-dual cuspidal automorphic representation for GL(4) that is of symplectic type and that does not admit a non-trivial self-twist. Recall that $\Lambda ^2 \Pi$ has the isobaric decomposition $\widetilde{r_5} \boxplus \omega$. Then $\widetilde{r_5}$ is cuspidal.
\end{lemma}

\begin{proof}
Let us write $r := \widetilde{r_5}$. Now $r$ is an isobaric automorphic representation for GL(5), so if it is not cuspidal, it must have an isobaric summand of degree one or two. To address the latter case, we note that by Proposition 4.2 of ~\cite{AR11}, $\Lambda ^2 \Pi$ does not have a degree two summand, and so the same holds for $r$.

We turn to the issue of a degree one summand. 
Recall that the Hecke character $\eta$ is defined such that $\widetilde{\Pi}\simeq \Pi \otimes \eta$.
Now consider, for an arbitrary Hecke character $\chi$, 
\begin{align*}
L(s, \Pi \times \widetilde{\Pi}\otimes \chi) &= L(s, \Pi \times \Pi \otimes \eta \chi)\\
&= L(s, \Pi, {\rm Sym}^2 \otimes \eta \chi) L(s, \Lambda ^2 \Pi \otimes \eta \chi).
\end{align*}
If $\chi \neq 1$, the left-hand side does not have a pole at $s=1$ since $\Pi$ does not admit a non-trivial self-twist. Then by Shahidi~\cite{Sh97}, $L(s, \Pi, {\rm Sym}^2 \otimes \eta \chi)$ is non-zero at $s=1$, so $L(s, \Lambda ^2 \Pi \otimes \eta \chi)$ has no pole at that point. Thus $\Lambda ^2 \Pi$ has no GL(1) isobaric summand other than possibly $\eta ^{-1}$, and the same applies for $r$.
Now $\Lambda ^2 \Pi$ has at least one degree one summand, namely $\omega$, so we must have $\omega = \eta ^{-1}$. When $\chi$ is trivial, the left-hand side of the equation above has a simple pole at $s=1$, which means that, following the reasoning above, $\eta ^{-1}$ only appears once as a degree one summand of $\Lambda ^2 \Pi$; therefore, $r$ has no GL(1) summands.
We conclude that $r$ is cuspidal.
\end{proof}

Now we can address Cases 1 and 2(b)(ii). Since $\Pi$ does not admit a non-trivial self-twist, $\widetilde{r_5}$ is cuspidal. Thus the bounds on $a_q(\Lambda ^2 \Pi)$ imply that $|a_q(\widetilde{r_5})| > 2 m - 3$, and Theorem~\ref{dr} implies that the set $T$ at which this can occur has bound 
\begin{align}\label{est2}
   \overline{\delta} (T) \leq \frac{1}{(2m-3)^2}.
\end{align}

Combining estimates~(\ref{est1}) and~(\ref{est2}) finishes the proof of the remaining cases of Theorem~\ref{t1}.

\subsection{Self-dual case}

We fix our cuspidal automorphic representation $\Pi$ to be self-dual and improve our estimates.

As always, let $q$ be a prime at which $\Pi$ is unramified. Being self-dual, the central character of $\Pi$ is either trivial or quadratic, and thus the determinant $w$ associated to the conjugacy class $A_q(\Pi)$ will either be $+1$ or $-1$.

\subsubsection{Trivial character}
Here we set $w = + 1$, then
$A_q(\Pi) = A_q(\widetilde{\Pi})$ implies that 
\begin{align*}
  A_q(\Pi) = \{\alpha, \alpha ^{-1}, \beta , \beta ^{-1}\}.
\end{align*}
Now
\begin{align*}
  A_q(\overline{\Pi})= \{\overline{\alpha}, \overline{\alpha ^{-1}}, \overline{\beta}, \overline{\beta ^{-1}}\}= \{\alpha ^{-1}, \alpha, \beta ^{-1}, \beta\} = A_q(\widetilde{\Pi}).
\end{align*}

We define $T = T (r)$ as before. For $q \in T$, there are three possible consequences for the structure of the Langlands conjugacy class: 

\begin{itemize}
  \item \textbf{Case I:} Here $\alpha = \overline{\alpha}$ and $\beta = \overline{\beta}$. We write $\alpha = u q^t$, $\beta = v q^s$, where $u,v \in \{1,-1\}$ and $t,s$ are non-negative (and at least one of them is positive). Then 
    \begin{align*}
      A_q(\Pi) = \{uq^t, uq^{-t}, vq^s, vq^{-s}\}.
    \end{align*}
\item \textbf{Case II:} We have $\alpha = \overline{\alpha}$ and $\beta = \overline{\beta ^{-1}}$, $\alpha$ is real and $\beta$ has absolute value 1.
  \begin{align*}
   A_q(\Pi) = \{uq^t, uq^{-t}, \beta, \beta ^{-1}\} 
  \end{align*}
where $t$ is positive and $u = \pm 1$.
\item \textbf{Case III:} We now assume that $\alpha \neq \overline{\alpha}$.

If $\alpha = \overline{\alpha^{-1}}$, then $\beta = \overline{\beta}$ or $\overline{\beta^{-1}}$. If the latter is true, then both $\alpha$ and $\beta$ have absolute value 1, and $q$ is a Ramanujan prime, so we shall exclude this possibility. If the former is true (i.e., $\beta = \overline{\beta}$), then we are in Case II.
Therefore, we shall exclude these two possibilities.

The remaining option is that (without loss of generality) $\alpha = \overline{\beta}$. We have 
\begin{align*}
  A_q(\Pi) = \{\alpha, \alpha^{-1}, \overline{\alpha}, \overline{\alpha^{-1}}\}
\end{align*}
\end{itemize}

For Cases I and III, it will be useful to consider the exterior square lift.\\

Case I:
\begin{align*}
  A_q(\Lambda^2\Pi) = \{1, uvq^{t+s}, uvq^{t-s}, uvq^{-t+s}, uvq^{-t-s}, 1\}. 
\end{align*}
Since $uv$ is real, 
 either $a_q(\Lambda^2\Pi) \geq 2m + 2$ or $a_q(\Lambda^2\Pi) \leq -2m + 2$. \\

Case III:
\begin{align*}
  A_q(\Lambda^2 \Pi) = \{1, |\alpha|^2, \alpha/\overline{\alpha}, \overline{\alpha}/\alpha, |\alpha|^{-2}, 1\}. 
\end{align*}
The associated Hecke eigenvalue is 
\begin{align*}
  a_q(\Lambda^2 \Pi) = |\alpha|^2 + |\alpha|^{-2} + (\alpha + \overline{\alpha})^2 / |\alpha|^2 \geq m ^2 - 2.
\end{align*}

In Cases I and III, $|a_q(\Lambda^2 \Pi)| > m$. Now recall that $\Lambda^2 \Pi \simeq \widetilde{r_5} \boxplus \omega$, where $\widetilde{r_5}$ is an automorphic representation for GL(5) and $\omega$ is a Hecke character. So we have that $|a_q (\widetilde{r_5})| > m - 1$.\\

Case II:
Note that 
\begin{align*}
a_q(\Pi) = u Q_t + (v + \overline{v}) 
\end{align*}
and, for the exterior square transfer, we see that 
\begin{align*}
A_q(\Lambda ^2 \Pi) &= \{1, uv q^t, u \overline{v}q^t, uv q ^{-t}, u \overline{v} q ^{-t},1\}\\
a_q(\Lambda ^2 \Pi) &= 2 + u (v + \overline{v})Q_t.
\end{align*}

If $u (v + \overline{v}) \geq -1$, then 
\begin{align*}
|a_q (\Pi)| > m - 1,
\end{align*}
whereas if $u (v + \bar{v}) < -1$, 
\begin{align*}
a_q(\Lambda ^2 \Pi) < 2-m,
\end{align*}
and so 
\begin{align*}
a_q(\widetilde{r_5}) < 1-m
\end{align*}
since $\omega = 1$.

So for the case of $\Pi$ with trivial central character, we can now combine our estimates to obtain   
\begin{align*}
\overline{\delta}(T) \leq  \frac{2}{(m-1)^2}.
\end{align*}

\subsubsection{Quadratic character}
Here, we need to consider both possible values of $w$. When $w = -1$, for $q \in T$, a similar analysis to above leads to a unique possibility:
\begin{align*}
A_q(\Pi) = \{u q ^t, u q ^{-t}, 1, -1\},
\end{align*}
thus
\begin{align*}
  |a_q(\Pi)| = Q_t > m.
\end{align*}

Combining estimates from both possibilities for $w$, we obtain 
\begin{align*}
\overline{\delta}(T) \leq   {\rm min}\left\{ \frac{1}{2}, \frac{2}{(m-1)^2}  \right\}   + \frac{1}{m ^2} .
\end{align*}

\section{Remaining Cases}\label{T}

Here, we address Theorem~\ref{t3}:
\begin{thm}
  Let $S$ be the set of places at which the Ramanujan conjecture holds. If $\Pi$ is a unitary cuspidal automorphic representation $\Pi$ for GL(4)/$F$ that is of type (T), then 
 \begin{align*}
\underline{\delta}(S) \geq \frac{33}{70}.
 \end{align*}
More specifically,
\begin{itemize}
  \item If $\Pi$ is an Asai transfer of a dihedral representation for GL(2), 
then we have $\underline{\delta}(S) = 1$.
  \item If $\Pi$ is automorphically induced from GL(2), then $\underline{\delta}(S) \geq 33/70$.
  \item If $\Pi$ is the automorphic tensor product of two cuspidal automorphic representations for GL(2), then $\underline{\delta}(S) \geq 33/35$.
\end{itemize}
\end{thm}

These are Cases (a), (b), and (d) of Theorem~\ref{ar}. Although we do not have cuspidality of the exterior square lift, we can instead make use of known results for cuspidal automorphic representations for GL(2), which allows us to obtain much stronger bounds compared to those of Theorems~\ref{t1} and~\ref{t2}. We briefly comment on these cases below.\\

\textbf{Case (a):} let $\Pi$ be a cuspidal automorphic representation for GL(4) such that $\Pi = \pi_1 \boxtimes \pi_2$, where $\pi_1, \pi_2$ are cuspidal automorphic representations for GL(2). Fix a finite place $v$ at which both $\pi_1$ and $\pi_2$ are unramified. We write 
\begin{align*}
  A_v(\pi_1) &= \{\alpha,\beta\},\\
  A_v(\pi_2) &= \{\gamma, \delta\}.
\end{align*}
Then 
\begin{align*}
  A_v(\Pi) = \{\alpha \gamma, \alpha \delta, \beta \gamma, \beta \delta\}. 
\end{align*}
By Kim--Shahidi~\cite{KS02}, for any cuspidal automorphic representation for GL(2), the set of places at which the Ramanujan conjecture holds has a lower Dirichlet density of at least 34/35. Therefore, for $\Pi$, the set of places at which the Ramanujan conjecture holds has a lower Dirichlet density of at least 33/35.\\

\textbf{Case (b):}
Let $\Pi$ be the Asai transfer of a dihedral cuspidal automorphic representation for GL(2). At all the places for which $\Pi$ is unramified, the associated Satake parameters will have absolute value 1 since they arise from the parameters of the corresponding dihedral representation, which is known to satisfy the Ramanujan conjecture. \\

\textbf{Case (d):}
Here we consider the cuspidal automorphic representation ${\rm AI}(\pi)$ for $GL_4(\A_F)$ which is the automorphic induction of $\pi$ for $GL_2(\A_E)$. Since the field extension $E/F$ is of degree 2, there are two possibilities for the shape of the local $L$-factor. If a place $v$ of $F$ splits in $E$, then we have (note that $\alpha_u, \beta_u$ denote the Satake parameters for $\pi$ associated to a place $u$) 
\begin{align*}
  L(s, {\rm AI}(\pi)_v) =& \prod_{w_1,w_2 \mid v} L(s, \pi_{w_i}) \\
   =& (1 - \alpha _{w_1}Nv^{-s})^{-1}
 (1 - \beta _{w_1}Nv^{-s})^{-1}\\
&\cdot (1 - \alpha _{w_2}Nv^{-s})^{-1}(1 - \beta _{w_2}Nv^{-s})^{-1}
\end{align*}
whereas if $v$ does not split in $E$, we have 
\begin{align*}
  L(s, {\rm AI}(\pi)_v) =& \prod_{w \mid v} L(s, \pi_w)\\
   =& (1 - \alpha _{w}Nv^{-2s})^{-1}
 (1 - \beta _{w}Nv^{-2s})^{-1}.
\end{align*}
Since, by~\cite{KS02}, there exists a set of places of lower Dirichlet density 34/35 at which $\pi$ satisfies the Ramanujan conjecture, we see that for the first equation above, we have a set of places of relative lower Dirichlet density 33/35 at which the Ramanujan conjecture holds. For the second equation, $w$ is inert over $v$, the density result for $\pi$ does not provide us with any information for these places. Overall, we therefore have established a set of places of lower Dirichlet density 33/70 at which the Ramanujan conjecture holds.

\section{GL(3)}\label{GL3}

In this section, we briefly examine what our approach would give for unitary cuspidal automorphic representations for GL(3)/$F$, where $F$ is a number field. By Ramakrishnan~\cite{Ra04b}, given $\pi$ for GL(3)/$\Q$ the set of Ramanujan primes is infinite; however, it is not known if it is of positive density. As mentioned in that paper, this result can be extended to any number field for which the associated Dedekind zeta function has no real zeros in the interval $(0,1)$.  The aim here is to obtain a bound for the size of the Satake parameters that applies to a large set (i.e., of positive density) of primes. We then briefly discuss the self-dual case.

We now prove Theorem~\ref{t4}:

\begin{thm}
Let $\pi$ be a cuspidal automorphic representation for GL(3)/$F$, where $F$ is a number field.
Fix $r \geq 1$, and denote by $S_r$ the set of primes $p$ where 
\begin{align*}
  \left(\frac{r + 1}{2}\right) - \sqrt{\left(\frac{r + 1}{2}\right)^2 -1}  \leq |\alpha_{i,p}| \leq   \left(\frac{r + 1}{2}\right) + \sqrt{\left(\frac{r + 1}{2}\right)^2 -1}
\end{align*}
for all $i$.
Then 
  \begin{align*}
    \underline{\delta}(S_r) \geq 1 - \frac{1}{r^2}.
  \end{align*}
\end{thm}

\begin{proof}
Note that the first part of the proof is the same as in~\cite{Ra04b}.

We let $q$ denote both a prime of $F$ and its norm, as in previous sections.

Fix a prime $q$ at which $\pi$ is unramified. We shall express the associated Langlands conjugacy class $A_q(\pi)$ as $\{\alpha, \beta, \gamma\}$.
We make use of the fact that $A_q(\widetilde{\pi}) = A_q(\overline{\pi})$: 
\begin{align*}
\{\alpha ^{-1}, \beta ^{-1}, \gamma ^{-1}\} = \{\overline{\alpha}, \overline{\beta}, \overline{\gamma}\}. 
\end{align*}

Now let us consider the possibilities:\\
If $\alpha ^{-1} = \overline{\alpha}$, then we are left to consider
\begin{align*}
  \{\beta ^{-1}, \gamma ^{-1}\} = \{\overline{\beta}, \overline{\gamma}\}.
\end{align*}
If $\beta ^{-1} = \overline{\beta}$, then the same holds for $\gamma$ (i.e., $\gamma ^{-1} = \overline{\gamma}$) and so $q$ must be a Ramanujan prime.
If $\beta ^{-1} \neq \overline{\beta}$, then $\beta ^{-1} = \overline{\gamma}$ and we shall write $\beta = u q ^t$ where $u$ is unitary and $t$ is positive (without loss of generality). So 
\begin{align*}
  A_q(\pi) = \{\alpha, uq ^t, uq ^{-t}\}.
\end{align*}

If $\alpha ^{-1} \neq \overline{\alpha}$, then (without loss of generality) $\alpha ^{-1} = \overline{\beta}$.
We write $\alpha = uq ^t$ and $\beta = u q ^{-t}$. This leaves $\gamma ^{-1} = \overline{\gamma}$, and so $A_q(\pi)$ has the same structure as above.

Since, for any cuspidal automorphic representation $\pi$ for GL(n), we have
\begin{align*}
  \underline{\delta}(\{v \mid |a_v(\pi)|\leq r\}) \geq 1-\frac{1}{r^2},
\end{align*}
we observe that the inequality 
\begin{align*}
  q ^t + q ^{-t} \leq r+1
\end{align*}
holds for a set of lower Dirichlet density $1- 1/r^2$. The inequality above implies that 
\begin{align*}
  q^t \leq  \left(\frac{r + 1}{2}\right) + \sqrt{\left(\frac{r + 1}{2}\right)^2 -1},
\end{align*}
completing the proof.
\end{proof}

In the GL(4) case, we had imposed the condition of self-duality in order to improve our bounds. One might ask if the same can be done here. In fact, the self-dual case is much simpler:
One knows that a self-dual cuspidal automorphic representation $\Pi$ for GL(3) is actually a twist of an adjoint lift of a cuspidal automorphic representation for GL(2). Applying results of Kim--Shahidi, we would then obtain a bound of 
\begin{align*}
\underline{\delta}(S(\Pi,1)) \geq \frac{34}{35}, 
\end{align*}
where $S(\Pi,1)$ is defined as before, namely the set of Ramanujan primes for $\Pi$.

\bibliography{mybib}{}

\def\cprime{$'$} \def\cprime{$'$}
\begin{thebibliography}{Ram04b}

\bibitem[AC89]{AC89}
James Arthur and Laurent Clozel.
\newblock {\em Simple algebras, base change, and the advanced theory of the
  trace formula}, volume 120 of {\em Annals of Mathematics Studies}.
\newblock Princeton University Press, Princeton, NJ, 1989.

\bibitem[AR11]{AR11}
Mahdi Asgari and A.~Raghuram.
\newblock A cuspidality criterion for the exterior square transfer of cusp
  forms on {${\rm GL}(4)$}.
\newblock In {\em On certain {$L$}-functions}, volume~13 of {\em Clay Math.
  Proc.}, pages 33--53. Amer. Math. Soc., Providence, RI, 2011.

\bibitem[BB11]{BB11}
Valentin Blomer and Farrell Brumley.
\newblock On the {R}amanujan conjecture over number fields.
\newblock {\em Ann. of Math. (2)}, 174(1):581--605, 2011.

\bibitem[Clo91]{Cl91}
Laurent Clozel.
\newblock Repr\'esentations galoisiennes associ\'ees aux repr\'esentations
  automorphes autoduales de {${\rm GL}(n)$}.
\newblock {\em Inst. Hautes \'Etudes Sci. Publ. Math.}, (73):97--145, 1991.

\bibitem[Kim03]{Ki03}
Henry~H. Kim.
\newblock Functoriality for the exterior square of {${\rm GL}_4$} and the
  symmetric fourth of {${\rm GL}_2$}.
\newblock {\em J. Amer. Math. Soc.}, 16(1):139--183 (electronic), 2003.
\newblock With appendix 1 by Dinakar Ramakrishnan and appendix 2 by Kim and
  Peter Sarnak.

\bibitem[Kri12]{Kr12}
M.~Krishnamurthy.
\newblock Determination of cusp forms on {$GL(2)$} by coefficients restricted
  to quadratic subfields (with an appendix by {D}ipendra {P}rasad and {D}inakar
  {R}amakrishnan).
\newblock {\em J. Number Theory}, 132(6):1359--1384, 2012.
\newblock With an appendix by Dipendra Prasad and Dinakar Remakrishnan.

\bibitem[KS02]{KS02}
Henry~H. Kim and Freydoon Shahidi.
\newblock Cuspidality of symmetric powers with applications.
\newblock {\em Duke Math. J.}, 112(1):177--197, 2002.

\bibitem[LRS99]{LRS99}
Wenzhi Luo, Ze{\'e}v Rudnick, and Peter Sarnak.
\newblock On the generalized {R}amanujan conjecture for {${\rm GL}(n)$}.
\newblock In {\em Automorphic forms, automorphic representations, and
  arithmetic ({F}ort {W}orth, {TX}, 1996)}, volume~66 of {\em Proc. Sympos.
  Pure Math.}, pages 301--310. Amer. Math. Soc., Providence, RI, 1999.

\bibitem[Ram97]{Ra97}
Dinakar Ramakrishnan.
\newblock On the coefficients of cusp forms.
\newblock {\em Math. Res. Lett.}, 4(2-3):295--307, 1997.

\bibitem[Ram00]{Ra00}
Dinakar Ramakrishnan.
\newblock Modularity of the {R}ankin-{S}elberg {$L$}-series, and multiplicity
  one for {${\rm SL}(2)$}.
\newblock {\em Ann. of Math. (2)}, 152(1):45--111, 2000.

\bibitem[Ram02]{Ra02}
Dinakar Ramakrishnan.
\newblock Modularity of solvable {A}rtin representations of {${\rm
  GO}(4)$}-type.
\newblock {\em Int. Math. Res. Not.}, (1):1--54, 2002.

\bibitem[Ram04a]{Ra04}
Dinakar Ramakrishnan.
\newblock Algebraic cycles on {H}ilbert modular fourfolds and poles of
  {$L$}-functions.
\newblock In {\em Algebraic groups and arithmetic}, pages 221--274. Tata Inst.
  Fund. Res., Mumbai, 2004.

\bibitem[Ram04b]{Ra04b}
Dinakar Ramakrishnan.
\newblock Existence of {R}amanujan primes for {GL}(3).
\newblock In {\em Contributions to automorphic forms, geometry, and number
  theory}, pages 711--717. Johns Hopkins Univ. Press, Baltimore, MD, 2004.

\bibitem[Sha97]{Sh97}
Freydoon Shahidi.
\newblock On non-vanishing of twisted symmetric and exterior square
  {$L$}-functions for {${\rm GL}(n)$}.
\newblock {\em Pacific J. Math.}, (Special Issue):311--322, 1997.
\newblock Olga Taussky-Todd: in memoriam.

\end{thebibliography}
\bibliographystyle{alpha}

\end{document}